\documentclass[12pt,leqno,amsfonts,amscd]{amsart}
\usepackage{epsfig}
\usepackage{amsmath}
\usepackage{amsfonts}
\usepackage{amssymb}
\usepackage{color}
\usepackage{hyperref}

\newtheorem{theorem}{Theorem}[section]
\newtheorem{prop}[theorem]{Proposition}
\newtheorem{cor}[theorem]{Corollary}
\theoremstyle{definition}
\newtheorem{definition}[theorem]{Definition}
\newtheorem{example}[theorem]{Example}

\theoremstyle{remark}

\newtheorem{remark}[theorem]{Remark}

\numberwithin{equation}{section}

\newcommand{\NN}{{\mathbb N}}

\newcommand{\RR}{{\mathbb R}}

\newcommand{\out}[1]{\ }

\let\cal=\mathcal
\renewcommand{\phi}{\varphi}

\begin{document}
\title[On the existence of the biharmonic  Green kernels
]{On the existence of the biharmonic  Green  kernels
and  the adjoint  biharmonic functions}

\author[A. Aslimani]{Abderrahim Aslimani}
\address{University of Mohammed V
\\Department of Mathematics
\\Faculty of Sciences
\\P.B. 1014, Rabat
\\Morocco}
\email{slimonier.math@gmail.com}

\author[I. El Ghazi]{Imad El Ghazi}
\address{Universit\'e Mohammed V
\\D\'epartement de Math\'ematiques
\\Facult\'e des Sciences
\\B.P. 1014, Rabat
\\Morocco}
\email{el.imad.imad@gmail.com}

\author[M. El Kadiri]{Mohamed El Kadiri}
\email{elkadiri30@hotmail.com}

\author[S. Haddad]{Sabah Haddad}
\address{CRMEF, Rabat-Akkari, Morocco}
\email{sabahhaddad24@gmail.com}



\subjclass[2000]{31B30, 31C40, 31D05}
\keywords{Harmonic function,  Biharmonic function,
Harmonic space, Biharmonic space, Green kernel, Potential,
Biharmonic Green kernel, Adjoint harmonic function, Adjoint biharmonic function.}

\begin{abstract} We study the existence and the
regularity of the biharmonic Green kernel in a Brelot biharmonic space
whose associated harmonic spaces have  Green
kernels. We  show by  some examples that this kernel does not always exist.
We then introduce and study the adjoint of the given biharmonic space. This
study was initiated  by Smyrnelis, however, it seems that several results were
incomplete and we clarify them here.
\end{abstract}

\maketitle

\section{Introduction}\label{section1}

The study of the Laplace equation  $\Delta u=0$ on an open subset
$U$ of $\RR^n$  led to various axiomatic theories of harmonic
functions  (Brelot, Bauer, Constantinescu-Cornea and Doob) that
apply more generally  to partial differential equations of the
type $Lu=0$, where  $L$ is an elliptic or parabolic second-order linear
differential operator. Such theories do not apply to the equations of higher order as the
classical biharmonic equation $\Delta^2 u=\Delta(\Delta u)=0$ in
$U$ or, equivalently, to the system $\Delta u=-v$, $\Delta v=0$.
This led
Smyrnelis to develop and study in \cite{S1,S2}  an
axiomatic theory within a larger framework which apply  to the partial
differential equations $L_1L_2u=0,$ where $L_1$ and $L_2$ are two
elliptic or parabolic second-order linear differential operators
on an open subset of $\RR^n$ or of a Riemannian manifold. In
this theory, a biharmonic space is the data of a locally compact
space $\Omega$ with a sheaf $\cal H$ of vector spaces of pairs of
continuous real functions (called the biharmonic pairs)
satisfying certain axioms (cf. Section \ref{section2}). To such space are
associated two harmonic spaces $(\Omega,\cal H_1)$ and $(\Omega,\cal
H_2)$ (cf. \cite[Theorem 1.29, p. 57]{S1}). Then one introduces the
superharmonic and potential pairs and  extends to this
framework the notion of balayage,  Dirichlet problem and other
concepts of classical or axiomatic theories of harmonic functions. If $L_1$ and $L_2$ are
two elliptic or parabolic second-order linear differential operators
on an open subset $\Omega$ of $\RR^n$, then $\Omega$ endowed with the
sheaf $\cal H$ on $\Omega$ defined by $\cal H(U)=\{(u,v)\in \cal
C^2(U): L_1u=-v, \ L_2v=0\}$ for every open subset $U$ of $\Omega$,
is a biharmonic space in the sense of Smyrnelis.

Let  $(\Omega,{\cal H})$ be a strong biharmonic space (that is, there is
a ${\cal H}$-potential $(p,q)>0$ on $\Omega$) whose associated harmonic
spaces $(\Omega,\cal H_1)$ and $(\Omega,\cal H_2)$ are Brelot
spaces. We assume that the ${\cal H}_2$-potentials of the same harmonic support reduced to one point
(that is, $\cal H_2$-harmonic outside
a set reduced to a single point) are proportional. According to \cite[Proposition 22.1, p. 507]{He}
there exists a continuous function $y\mapsto q_y$ from
$\Omega$ into a compact base $B$ of the  cone $\cal S_2^+(\Omega)$ of
 ${\cal H}_2$-superharmonic functions $\ge 0$ on $\Omega$ endowed with the
topology of R-M. Herv\'e cf. \cite[Chap. IV]{He}. The  function $G$ defined on
$\Omega^2=\Omega\times \Omega$ by $G(x,y)=q_y(x)$ is l.s.c. on $\Omega^2$
 and continuous outside the diagonal of $\Omega^2$ ($G$ will be called a (or the)
Green kernel of the harmonic space$(\Omega,\cal H_2)$). Under the assumption $(1,1)\in {\cal
H}^*(\Omega)$ Smyrnelis affirms in \cite{S4} that, for every $y\in \Omega$, the pure hyperharmonic
function $p'_y$ of order 2 associated with $q_y$ (see Section \ref{section7}) is $\cal
H_1$-superharmonic function and that the  function $y\mapsto p'_y$
is continuous from $\Omega$ into $\cal S_1^+(\Omega)$. This result was
used by Smyrnelis to define in  \cite{S5} adjoint biharmonic functions. It is also stated in a
remark at the end of  \cite{S4}, that there is a
biharmonic pair $(h,k)$ such that $h>0$ and $k>0$ on
$\Omega$. By several examples stated in Section
\ref{section7} below, we  see that these statements are not always true, and therefore the results of
\cite{S5} are incomplete without additional assumptions
ensuring the existence of a biharmonic  Green kernel (or of a second Green
kernel). In Section \ref{section8} of
the present work we also consider  adjoint biharmonic functions. We then define
the adjoint biharmonic space $(\Omega,{^*{\cal H}})$ of $(\Omega, \cal H)$ under the following
supplementary conditions:\\

{\bf 1.} For any $y\in \Omega$, the pure hyperharmonic function $p'_y$ of
order two associated with
$q_y$
is $\cal H_1$-superharmonic,

{\bf 2.} For any $x\in \Omega$, the function $y\mapsto p'_y(x)$ is continuous on $\Omega\smallsetminus\{x\}$.\\

\noindent
It is a biharmonic space whose associated
harmonic spaces are  respectively the adjoint harmonic spaces
$(\Omega,{^*\cal H_2})$ and $(\Omega,{^*\cal H_1})$ of the harmonic
spaces $(\Omega,\cal H_2)$ and $(\Omega,\cal H_1)$ respectively, and
in which the pure hyperharmonic function of order 2 associated with
$p_y^*$ is the function ${p'_y}^*=p'_\centerdot(y)$ for every $y\in \Omega$,
where $y\mapsto p_y$ is a continuous map from $\Omega$ into $\cal S_1^+(\Omega)$
which associates to any $y\in \Omega$ an
$\cal H_1$-potential with harmonic support $\{y\}$.
We also show that, conversely, if there is a  biharmonic space
$(\Omega, \cal G)$ whose associated biharmonic spaces  are
$(\Omega,{^*\cal H_2})$ and $(\Omega, {^*\cal H_1})$ respectively, and where the
pure hyperharmonic function  of order 2 associated with $p_y^*$ is the
function ${p'_y}^*=p'_\centerdot(y)$ for every $y\in \Omega$
(where $p^*_y$ is the $\cal H_1$-adjoint superhharmonic defined by $p_y(x)=p_x(y)$), then
the above condition {\bf 2} is satisfied and the coupling kernel
of $(\Omega,\cal G)$ is equal to the adjoint
kernel $V^*$ of the coupling kernel $V$ of $(\Omega,\cal G)$
(see definition of $V^*$ in Section \ref{section8}). We also show that the (function) kernel
$H$ defined on $\Omega^2$ by $H(x,y)=p'_y(x)$, called the biharmonic
Green kernel (or the second Green kernel) of the  biharmonic space $(\Omega,\cal H)$, is regular
in the sense that it is l.s.c. on $\Omega^2$ and continuous outside the diagonal of  $\Omega^2$.

In this paper we also give results in the axiomatic theory of biharmonic spaces
and clarify some incomplete results of Smyrnelis. More precisely, we shall prove the following:

1. In a biharmonc space associated with Green harmonic space the biharmonic Green kernel may not exist.

2. The existence of biharmonic space associated with two harmonic spaces and
a given coupling kernel.

3. Contrarily to the case of harmonic space, a biharmonic pair $>0$ (or a strict biharmonic function $>0$) in
a strong biharmonic space may not exist.

4. A definition of the notion of biharmonic adjoint of given Green biharmonic space is possible.

5. The biharmonic Green kernel in a biharmonic space, when it exists, is regular under some simple conditions.\\

{\bf Notations}. Let  $\Omega$ be a locally compact space with
countable base. If $A$ is a subset of $\Omega$, we denote by
$\overline A$ the closure of  $A$ in the Alexandroff  compactification
$\overline \Omega$ of $\Omega$, and by $\partial A$ the boundary of
$A$ in $\overline \Omega$. By function on $A$ we mean a function
on $A$ with values in $[-\infty,+\infty]$. Let $U$ be an open subset of
$\Omega$, the set of Borel functions  on $U$ is denoted by
$\cal B(U)$ and, if $f$ is a function defined on $U$, we denote by
$\widehat f$ the l.s.c. regularization of $f$. Recall that $\widehat f$
is defined by $\widehat f(x)=\lim\inf_{y\to x} f(y)$ for every $x\in
U$ and that $\widehat f$  is the greatest l.s.c. of $f$
on $U$. We denote by $\cal C_c(\Omega)$ the (real) vector space
of continuous and finite functions on $\Omega$ with compact support.
For every continuous function $\varphi$ on $\Omega$,
we denote by $S(\varphi)$ the support of $\varphi$. The order on the
set of  pairs of functions on a subset $A$ of $\Omega$ is defined by
$$(f_1,g_1)\ge (f_2,g_2) \iff f_1\ge f_2 \ {\rm et} \
f_2\ge g_2 \ {\rm on} \ A.$$ For the sake of simplicity, we shall
write $(f,g)\ge 0$ for  $(f,g)\ge (0,0)$ and $(f,g)>0$ on $A$ for $f
\ {\rm and} \ g>0$ on $A$.

If $\cal A(E)$ is a set of functions defined on a set $E$, we denote by
$\cal A^+(E)$ the set of nonnegative functions  of $ \cal A(E)$.

\section{Preliminaries and notations}\label{section2}

Throughout this article,  $\Omega$ is a locally compact space
with countable base and   ${\cal U}$ (resp. ${\cal U}_c$)
is the  set of open subsets (resp. the set of relatively
compact open subsets) of $\Omega$. Let ${\cal H}$ be a
map which  associates to each  $U\in {\cal U}$  a
vector space of pairs of continuous real functions on $U$, compatible
in the sense that if $(u,v)\in {\cal H}(U)$ and $u=0$ on an open
subset $\omega \subset U$ then $v=0$ on $\omega$. The pairs of
${\cal H}(U)$
are termed biharmonic on $U$.\\

{\bf Axiom 1.} {\it ${\cal H}$ is a sheaf.}\\

This means that

1. If $U$ and $V$ are both open subsets of $\Omega$ such that
$U\subset V$ and if $(u,v)\in {\cal H}(V)$ then $(u_{|U},v_{|U})\in
{\cal H}(U)$.

2. If $(U_i)$ is a family of open subsets of  $\Omega$ and if
$(u,v)$ is a pair of real functions on $U=\bigcup_iU_i$  such that, for
every $i$, $(u_{|U_i},v_{|U_i})\in{\cal H}(U_i)$,
then $(u,v)\in {\cal H}(U)$.\\

An open subset $\omega \in {\cal U}_c$ is said to be ${\cal
H}$-regular, or just regular, if for every pair $(f,g)$  of real
continuous functions on ${\partial \omega}$, there exists a pair
$(u,v)\in {\cal H}(\omega)$ such that

i) $\lim _{x\to y}u(x)=f(y)$ and $\lim_{x\to y} v(x)=g(y)$ for every
$y\in {\partial \omega}$.

ii) If $(f,g)\ge (0,0)$, then $(u,v)\ge (0,0)$.

Such pair is then unique. It is denoted by
$H_\omega(f,g)=(H_\omega^1(f,g),H_\omega^2(f,g))$ and  called the
solution of the Riquier problem in $\omega$ for the boundary data
 $(f,g)$ on $\partial \omega$. If $\omega\in {\cal U}_c$ is ${\cal
H}$-regular, then for any $x\in \omega$, there exists a triple
$(\mu_x^\omega,\nu_x^\omega,\lambda_x^\omega)$ of
nonnegative Radon measures
on ${\partial \omega}$, called the triple of   biharmonic measures
of $\omega$ at the point $x$,  such that

$$H_\omega^1(f,g)(x)=\int f(y)d\mu_x^\omega(y)+\int g(y)d\nu_x^\omega(y)$$
and
$$H_\omega^2(f,g)(x)=\int g(y)d\lambda_x^\omega(y)$$
for any pair $(f,g)\in {\cal C}({\partial \omega})\times {\cal C}({\partial \omega})$.\\

{\bf Axiom 2}. {\it The regular open sets  form a base of the topology of $\Omega$.}\\

We denote by ${\cal U}_r$ the set of all $\cal H$-regular relatively
compact open subsets of $\Omega$. A pair $(u,v)$ of l.s.c. functions on $U\in
{\cal U}$ with values in $]-\infty,+\infty]$ is said to be ${\cal
H}$-hyperharmonic (on $U$), or just hyperharmonic when there is no risk of
confusion, if for any  $\omega\in {\cal U}_r$, $\omega\subset
{\overline \omega}\subset U$, and any $x\in \omega$, we have
$$\int^* ud\mu_x^\omega+\int^* vd\nu_x^\omega\le u(x)$$ and
$$\int^* vd\lambda_x^\omega \le v(x).$$

The set of ${\cal H}$-hyperharmonic pairs on an open subset $U$ of
$\Omega$ is denoted by ${\cal H}^*(U)$. It is easy to check that
${\cal H}^*(U)$ and ${\cal H}^{* +}(U)$ are convex cones.

For every  $U\in {\cal U}$ we set
$${\cal H}_1^*(U)=\{u: (u,0)\in {\cal H}^*(U)\}$$
and
$${\cal H}_2^*(U)=\{v: (+\infty,v)\in {\cal H}^*(U)\}.$$

Then ${\cal H}_1^*$ and ${\cal H}_2^*$ are two sheaves of cones of
l.s.c. functions.\\

Then we define  on $\Omega$ two  sheaves of vector spaces of
continuous real  functions  by putting $\cal H_1(U)=\cal
H_1^*(U)\cap(-\cal H_1^*(U))$ and  $\cal H_2(U)=\cal
H_2^*(U)\cap(-\cal H_2^*(U))$
for any open subset $U$ of $\Omega$.\\

{\bf Axiom 3.} {\it $(\Omega, {\cal H}_1)$ and $(\Omega,{\cal H}_2)$
are Bauer harmonic spaces.}\\

We say that the pair $(\Omega,\cal H)$ or that   $\Omega$, provided
with the sheaf ${\cal H}$, is a biharmonic space if the axioms 1, 2 and 3 are satisfied.\\

\begin{example}\label{example2.1} Let $N$ be an integer $\ge 1$, and let $L_1$, $L_2$
two elliptic or parabolic second-order linear differential operators, on an (nonempty) open subset $\Omega$  of $\RR^N$.
Then $\Omega$ endowed with the sheaf ${\cal H}$ defined by
$${\cal H}(U)=\{(h,k)\in {\cal C}^2(U)\times {\cal C}^2(U):
L_1h=-k, L_2k=0\}$$ for every open subset $U$ of $\Omega$, is a
biharmonic space.
\end{example}

A  biharmonic space  $(\Omega,\cal H)$ is  said to be a Brelot
biharmonic space if $\Omega$ is non-compact, connected and locally
connected and if the associated harmonic spaces  $(\Omega,\cal H_1)$
and $(\Omega,\cal H_2)$ are Brelot harmonic spaces.

The probabilistic aspect of biharmonic spaces was studied by Bouleau
in \cite{B}. Boboc and Bucur  \cite{BB} have shown that the
nonnegative hyperharmonic  pairs can be identified with the excessive
functions of a triangular resolvent on the space $\Omega \oplus
\Omega$, and then the theory of biharmonic spaces
reduces to that of $H$-cones. For more details on the theory
of biharmonic spaces, we refer to \cite{S1} and  \cite{S2}.\\

Let $U\in {\cal U}$ and $(u,v)\in {\cal H}^*(U)$. If $u$ is finite
on a dense subset in $U$ (and hence $v$ also is finite in a dense subset
in $U$), we say that the pair $(u,v)$ is ${\cal
H}$-\textit{superharmonic} or just \textit{superharmonic} on $U$. We denote by ${\cal
S}(U)$ the set of all superharmonic pairs on  $U$. Then ${\cal
S}(U)$, ${\cal S}^+(U)$ and $\cal H^+(U)$ are  convex cones.

Let $(u,v)$ be a $\cal H$-superharmonic pair on $\Omega$. It follows from
Axiom 1 that there is a greatest open subset $U$ of $\Omega$ on which $(u,v)$ is $\cal H$-harmonic.
The complement set of $U$ is called the \textit{biharmonic support} of $(u,v)$.

A pair $(p,q)\in {\cal S}^+(\Omega)$ is said to be a $\cal
H$-\textit{potential} (or just a \textit{potential} if there is no risk of confusion) if
$$\forall (h,k)\in {\cal H}^+(\Omega); (h,k)\le (p,q)\Rightarrow
h=k=0.$$

It is easy to see that if $(u,v)\in \cal S^+(\Omega)$ and if $u$ is
a $\cal H_1$-potential and $v$ is a $\cal H_2$-potential, then
$(u,v)$ is a  $\cal H$-potential. Conversely, if $(u,v)$ is a $\cal
H$-potential, then $u$ is a $\cal H_1$-potential and, as it will be
seen in the next section (see Remark \ref{cor3.9}), $v$ is
a $\cal H_2$-potential.

A biharmonic space $(\Omega,\cal H)$ is said to be strong, if there
exists a $\cal H$-potential $(p,q)>0$ on $\Omega$.
For $L_1=L_2=\Delta$, the Laplace operator, in the
Example \ref{example2.1},  the space $(\RR^N,\cal H)$ is strong if and only if
$N\ge 5$ (see \cite{EK1}). Any  relatively compact open subset $\omega$
of a biharmonic space, with  the biharmonic sheaf restricted to
$\omega$, is strong.

In order to simplify some statements in the next sections we introduce the following definition:

\begin{definition}
An harmonic space $(\Omega,\cal K)$ is called a Green harmonic space
if $(\Omega,\cal K)$ is a Brelot harmonic space satisfying the hypothesis of uniqueness,
that is, the potentials on $\Omega$
with the same harmonic support reduced to one point are
proportional.  A biharmonic space $(\Omega,\cal H)$ is called Green biharmonic space if its associated harmonic
spaces $(\Omega,\cal H_1)$ and $(\Omega,\cal H_2)$ are Green harmonic spaces.
\end{definition}

\section{Pure hyperharmonic pairs and kernels  associated with a  biharmonic space}\label{section3}

Let $(\Omega,{\cal H})$  be a biharmonic space and  $v$ a
nonnegative ${\cal H}_2$-hyperharmonic function on $\Omega$. The set
$$\cal U_0(v)=\{u\in \cal H_1^{*+}(\Omega): (u,v)\in \cal H_+^*(\Omega)\}$$
is not empty because the constant function $u\equiv +\infty$
belongs to $\cal U_0(v)$. Then, according to \cite[Lemma 11.6, p.
36]{S2} the function $u_0=\widehat {\inf} \ \cal U_0(v)\in {\cal
H}_1^{*+}(\Omega)$. It is the smallest nonnegative ${\cal
H}_1$-hyperharmonic function $u$ on $\Omega$ such that $(u,v)\in {\cal
H}^{*+}(\Omega)$. This function is called the \textit{pure hyperharmonic
function  of order} 2, or simply
the pure $\cal H_1$-hyperharmonic, associated with $v$. A pair $(u,v)\in {\cal
H}^{*+}(\Omega)$ is said to be \textit{pure} if $u$ is the pure hyperharmonic
function of order 2 associated with $v$.

The essential properties of
the pure  hyperharmonic pairs  can be found in \cite{B} and \cite{B2}.

\begin{remark}\label{remark3.1}
If the pair $(u,v)\in \cal S^+(\Omega)$ is pure, then $u$ is a $\cal
H_1$-potential. Indeed, $u$ is nonnegative $\cal H_1$-superharmonic.
Let $h$  be the $\cal H_1$-harmonic part in the  Riesz decomposition of
$u$, then the pair $(u-h,v)$ is nonnegative $\cal H$-superharmonic
in $\Omega$, whence $u\le u-h$ and therefore $h=0$, which clearly
proves that $u$ is a $\cal H_1$-potential.
\end{remark}

\begin{remark}\label{remark3.2}
Let $(p,q)$ be a pure  $\cal H$-superharmonic pair on $\Omega$ with $q>0$,
then $p$ is a strict potential. In fact, for any $\omega\in \cal U_r$, we have $p\ge \int
pd\mu_\centerdot^\omega+\int qd\nu_\centerdot^\omega>\int pd\mu_\centerdot^\omega$. Thus $p$ is
strict.
\end{remark}

Recall that a kernel  $V$ on a measurable space
$(E,\cal E)$ is a function $V\ge 0$ defined on  $E\times \cal E$
such that

1. For every $A\in \cal E$, the function $x\mapsto V(x,A)$ is
$\cal E$-measurable.

2. For every $x\in E$, the set function
$A\mapsto V(x,A)$ is a (nonnegative) measure  on $(E,\cal E)$.

Let  $V$ be a kernel on $(E,\cal E)$, then for any nonnegative
$\cal E$-measurable function $f$ on $E$, we denote by $Vf$ or $V(f)$ the function on $E$
defined by
$$Vf(x)=\int f(y)V(x,dy), \ \forall x\in E,$$
where the  integral is understood to be taken with respect to the measure $V(x,\centerdot)$.

Let $\mu$ be a measure on $(E,\cal E)$ and $f$ a nonnegative $\cal E\otimes \cal E$-measurable
function on
$E\times E$, then the function $(x,A)\mapsto \int_Af(x,y)d\mu(y)$ on $E\times \cal E$ is a
kernel on $(E,\cal E)$.

If $E$ is a topological space and $\cal E$ is the Borel
$\sigma$-algebra
of $E$, then a kernel on $(E,\cal E)$ is called a Borel kernel on  $E$.\\

Recall the following result of Bouleau \cite[Theorem 2.8, p.
208]{B} that will be very useful in the rest this paper:

\begin{theorem}\label{thm3.7}
Let $(\Omega,\cal H)$ be a strong biharmonic space.
There exists a unique Borel kernel $V$ on $\Omega$
such that

{\rm 1}. For any  function  $\varphi\in \cal C_c^+(\Omega)$,
the function $V\varphi$ is a $\cal H_1$-potential on $\Omega$, and
$\cal H_1$-harmonic on $\Omega\smallsetminus S(\varphi)$.

{\rm 2}. For every function $v\in \cal H_2^{*+}(\Omega)$, $Vv$ is the pure
 hyperharmonic function of order 2 associated with $v$.
\end{theorem}

The kernel $V$ will be called  the \textit{coupling kernel}  of the harmonic
spaces $(\Omega,\cal H_1)$ and $(\Omega, \cal H_2)$ in the
biharmonic space $(\Omega,\cal H)$ or simply
the \textit{coupling kernel} of $(\Omega,\cal H)$.\\

Let $(p,q)$ be a  pure finite and continuous ${\cal H}$-potential
$>0$ on $\Omega$. We have seen  in the Remark \ref{remark3.2} that  the
potential $p$ is strict. Then there exists by
\cite[Exercice 8.2.3, p. 198]{CC} (or \cite[Theorem 2, p.
362]{Me} for Brelot harmonic spaces) a unique Borel kernel $W$ on  $\Omega$ such that

1. For any  function  $\varphi\in \cal C_c^+(\Omega)$,
the function $W\varphi$ is an $\cal H_1$-potential on $\Omega$, and
$\cal H_1$-harmonic on $\Omega\smallsetminus S(\varphi)$.

2. $W1=p.$

Let $V$ be the coupling kernel of $(\Omega,\cal H)$,
and consider the Borel kernel
$V'$ on $\Omega$ defined by $V'f=V(qf)$ for any Borel function $f\ge 0$ on
$\Omega.$ Then $V'$ satisfies the above two properties 1. and 2. of $W$.
Hence, according to the uniqueness of the kernel $W$, we have
$W=V'$. From this result we deduce the following:

\begin{prop}\label{prop3.4}
For every  $\cal H_2$-hyperharmonic function $v\ge 0$ on $\Omega$,
the pure $\cal H_2$-hyperharmonic function associated with $v$ is
$u=W(\frac{v}{q})$.
\end{prop}

\begin{prop}\label{prop3.5} Let $(u,v)\in \cal S_+^*(\Omega)$ and let $u_0$ be the pure
$\cal H_1$-hyperharmonic function of order 2 associated with $v$. Then there
exists a nonnegative $\cal H_1$-hyperharmonic function $u_1$ in
$\Omega$ such that $u=u_0+u_1$.
\end{prop}

\begin{proof} Since $(u,v)$ is $\cal H$-superharmonic
then $u$ and $u_0$ are $\cal H_1$-superharmonic functions.
Suppose first that $u$ and $v$ are finite and
let $\omega\in \cal U_r$ be a $\cal H$-regular open  subset  and let $w$ be the
function defined on $\omega$ by $w=u-\int ud\mu_\centerdot^\omega+\int
u_0d\mu_\centerdot^\omega$. Then the pair $(w,v)$ is $\cal H$-superharmonic
on $\omega$. Furthermore we have $\lim\inf_{x\in \omega, x\to z}
w(x)\ge u_0(z)$ for every $z\in \partial \omega$. It therefore follows from
\cite[Proposition 1.21, p. 53]{S1} that the pair $(u_2,v)$ is $\cal
H$-superharmonic in $\Omega$, where $u_2$ is the function defined on
$\Omega$ by
$$u_2= \left\{\begin{array}{c}
w\wedge u_0 \text{ on } \omega\\
u_0 \text { on } \Omega \smallsetminus \omega,
\end{array}
\right.$$ hence $u_2\ge u_0$ and therefore $w\ge u_0$ in $\omega$. Then it follows that the
nonnegative function $u-u_0$ is $\cal H_1$-hyperharmonic
function $u_1$ on $\Omega$ and we have $u=u_0+u_1$.
For general $\cal H$-hyperharmonic pair $(u,v)$, there exists an increasing sequence
$((u_n,v_n))$ of finite $\cal H$-superharmonic  pairs such that $(u,v)=\sup_n(u_n,v_n)$.
By the above case, for each $n$ there exists a $\cal H_1$-superharmonic $t_n$ and a pure
pair $(u'_n,v_n)$ such that $u_n=u'_n+t_n$. By Theorem \ref{thm3.7},
the sequence $((u'_n,v_n))$ is increasing and the pair $(\sup_nu'_n, v)$ is pure.
It follows that $u=\sup_nu'_n+\sup_m\widehat{\inf}_{n\ge m}t_m$, which proves the proposition.
\end{proof}

\begin{cor}\label{cor3.6}
Suppose that the biharmonic space $(\Omega,\cal H)$ is strong. Then
there exists a  pure finite and continuous $\cal
H$-potential $(p,q)>0$ on $\Omega$.
\end{cor}

\begin{proof}  By \cite[Proposition 7.6, p. 3]{S1}, there exists a
finite continuous   $\cal H$-potential
pair $(p,q)>0$ on $\Omega$. Let $p_1$ be the pure
$\cal H_1$-hyperharmonic function  associated with $q$, there exists by the preceding proposition a
$\cal H_1$-superharmonic function $u\ge 0$ such that
$p=u+p_1$. Since $u$ and $p_1$ are l.s.c. and $p$ is finite and continuous on $\Omega$, we
deduce that $p_1$ is finite and continuous, and therefore $(p_1,q)$ is a pure finite and continuous $\cal H$-potential
on $\Omega$.
\end{proof}

\begin{cor}\label{cor3.3}
Assume that $(\Omega,\cal H)$ is a strong biharmonic space
and let $(u,v)$ be pure pair on $\Omega$. Then there exists an increasing sequence  $((u_n,v_n))$ of pure
finite and continuous $\cal H$-potentials such that $(u,v)=\sup_n(u_n,v_n)$.
\end{cor}

\begin{proof}
By \cite[Th\'eor\`eme 7.8, p. 4]{S2}, there exists an increasing
sequence $(p_n,q_n)$ of continuous  $\cal H$-potentials
on  $\Omega$ such that $(u,v)=\sup_n(p_n,q_n)$. For each $n$ let $p'_n$ be the
pure $\cal H_1$-hyperharmonic function associated with $q_n$. Then
by Proposition \ref{prop3.5} there exists $t_n\in \cal H_1^{*+}(\Omega)$ such
$p_n=t_n+p'_n$. Since $p_n$ is continuous and that $t_n$ and $p'_n$ are l.s.c., then $p'_n$ is continuous.
Moreover, it follows from Theorem \ref{thm3.7}
that the sequence $(p'_n)$ is increasing and
$(u,v)=\sup_n(p'_n,q_n)$.
\end{proof}

\begin{prop}\label{prop3.8}
Let $(u,v)\in \cal S^+(\Omega)$ be a pure pair and $U$  an open subset of $\Omega$.
If $v$ is $\cal H_2$-harmonic on $U$, then the pair $(u,v)$ is $\cal H$-harmonic on $U$.
\end{prop}

\begin{proof} Let $\omega\in \cal U_r$ such that $\overline \omega\subset U$. Then we have
$\int vd\lambda_x^\omega=v(x)$ for every $x\in \omega$, hence the pair
$$(s,v)= \left\{\begin{array}{c}
(\int ud\mu_\centerdot^\omega+\int vd\nu_\centerdot^\omega,v) \text{ on } \omega\\
(u,v) \text { on } \Omega \smallsetminus \omega
\end{array}
\right.$$
is $\cal H$-hyperharmonic on $\Omega$ by \cite[Proposition 4.2, p. 73]{S1}. Hence
$s\ge u$ and therefore $u=s=\int ud\mu_\centerdot^\omega+\int vd\nu_\centerdot^\omega$
on $\omega$. Since $\omega$ is arbitrary,  this proves that the pair
$(u,v)$ is biharmonic on $U$ by \cite[Proposition 5.4, p. 76]{S1}.
\end{proof}

\begin{cor}\label{cor3.9}
If a pair  $(u,v)\in \cal S^+(\Omega)$ is a $\cal
H$-potential, then $v$ is a $\cal H_2$-potential.
\end{cor}

\begin{proof} The function $v$ being $\cal H_2$-superharmonic,
let $k$ be the $\cal H_2$-harmonic part of $v$ in its Riesz decomposition
and $h$ the pure  $\cal H_2$-hyperharmonic associated with $k$. Then $h\le u$, hence
$(h,k)$ is $\cal H$-superharmonic and therefore biharmonic by the above proposition.
It follows that $(h,k)=0$, and hence $k=0$, which proves that $v$ is a $\cal H_2$-potential.
\end{proof}

\begin{cor}\label{cor3.10}
Let  $(h,k)$ be  a nonnegative $\cal H$-harmonic pair  and let $p$
be the $\cal H_1$-potential part of the   Riesz decomposition of the
 $\cal H_1$-superharmonic function $h$. Then  $(p,k)$ is
a pure $\cal H$-biharmonic pair.
\end{cor}

Let $(\Omega,\cal H)$ be a biharmonic space and suppose that the harmonic space $(\Omega,\cal H_1)$
is a Green space. Consider a continuous function $y\mapsto p_y$ from $\Omega$ into the
cone $\cal S_1^+(\Omega)$
endowed with the topology of R-M. Herv\'e,  which associates to any $y\in \Omega$
a $\cal H_1$-potential of  harmonic support $\{y\}$ (cf. \cite[Chap. IV]{He}
and \cite[Proposition 22.1, p. 507]{He}).  Hence, according to the
theorem of integral  representation of potentials (cf. \cite[Theorem
18.2, p. 482]{He}), there exists a unique nonnegative Radon measure
 $\mu$ on $\Omega$ such that $p=\int
p_yd\mu(y)$. In this case the kernel $W$ is then given by
$$Wf(x)=\int_\Omega p_y(x)f(y)d\mu(y)$$ for any
$x\in \Omega$ and any function $f\in \cal B^+(\Omega)$.
As it can be seen from the above considerations, the kernel $V$ of
Theorem \ref{thm3.7} is given by
$$Vf(x)=\int_{\Omega}\frac{p_y(x)}{q(y)}f(y)d\mu(y)$$
for any $x\in \Omega$ and any function $f\in \cal B^+(\Omega)$.
Thus we have established the following proposition:

\begin{prop}\label{prop3.9}
Let $(\Omega,\cal H)$ be a  strong  biharmonic
space and suppose that the  associated harmonic space $(\Omega,\cal H_1)$
is a Green space and with coupling kernel $V$. Consider a continuous function
$y\mapsto p_y$ from $\Omega$ into the cone $\cal
S_1^+(\Omega)$ of $\cal H_1$-superharmonic functions $\ge 0$ endowed with the
topology of R-M. Herv\'e, associating to any $y\in \Omega$ a
$\cal H_1$-potential of harmonic support $\{y\}$. Then there exists a
unique Radon
measure  $\mu\ge 0$ on $\Omega$ such that $Vf(x)=\int
p_y(x)f(y)d\mu(y)$ for any $x\in \Omega$ and any $f\in \cal B^+(\Omega)$.
\end{prop}

\section{biharmonic space associated with two strong harmonic spaces and a given coupling kernel}\label{section5}

Given two harmonic spaces $(\Omega,\cal H_1)$ and $(\Omega,\cal H_2)$, the first
space assumed to be strong, and a strict finite and continuous
$\cal H_1$-potential on $\Omega$, in \cite{S1}
it is constructed a biharmonic space $(\Omega,\cal H)$ whose harmonic spaces are $(\Omega,\cal H_1)$ and
$(\Omega,\cal H_2)$.  By using sections of continuous potentials, in \cite{Bouk} it is given
a more general method for constructing such biharmonic space and it is proved
a characterization of biharmonic spaces whose harmonic associated spaces are
$(\Omega,\cal H_1)$ and $(\Omega,\cal H_2)$.

Here we prove that there is a unique biharmonic space $(\Omega,\cal H)$ whose associated harmonic
spaces are $(\Omega,\cal H_1)$ and
$(\Omega,\cal H_2)$ with coupling kernel a given suitable kernel $V$ on $\Omega$.

Let  $\omega$  be a relatively compact open subset of $\Omega$, we
denote by $H_\omega^1(f)$ and $H_\omega^2(f)$ the solutions of the
Dirichlet problem in $\omega$ with respect to the harmonic spaces
$(\Omega,\cal H_1)$ and $(\Omega,\cal H_2)$ respectively, for the
continuous boundary data  $f$ on $\partial \omega$.
We also denote the operators of reduction (resp. swept
out) on $A\subset \Omega$ with respect to the   harmonic sheafs
$\cal H_1$ and $\cal H_2$  respectively by $R^{1,A}$ and $R^{2,A}$
(resp. $\widehat R^{1,A}$ and $\widehat R^{2,A}$).

The following result seems to be new in the axiomatic theory of biharmonic spaces:

\begin{theorem}\label{thm3.16}
Let $\cal H_1$ and $\cal H_2$ be two Bauer harmonic sheaves on
$\Omega$ and  $V$ a Borel kernel on $\Omega$ such that for any
$\varphi\in \cal C_c^+(\Omega)$, the function $V(\varphi)$ is a finite and continuous  $\cal H_1$-potential
$\cal H_1$-harmonic on $\Omega\smallsetminus S(\varphi)$.
Suppose  moreover that $\Omega$ has a base $\cal V$
of open subsets  which are both $\cal H_1$-regular and $\cal
H_2$-regular,  and that there exists a $\cal H_2$-potential $q>0$
such that $Vq$ is a $\cal H_1$-potential. Then there exists a unique
strong biharmonic space $(\Omega,\cal H)$ whose associated harmonic spaces
are $(\Omega,\cal H_1)$ and $(\Omega,\cal H_2)$, and  $V$ as
coupling kernel.
\end{theorem}

\begin{remark}\label{remark4.6}
For a Borel  kernel $V$ on an harmonic space $\Omega$, the property that for any continuous
function $f\ge 0$ on $\Omega$ with compact support, the function $Vf$
is a potential on $\Omega$, harmonic on $\Omega\smallsetminus S(f)$ is equivalent to the
property that for every $f\in \cal B^+(\Omega)$
which is bounded and equals zero outside a compact $K$ of $\Omega$, the function $Vf$ is
finite and continuous on $\Omega$ and harmonic on $\Omega\smallsetminus K$.
\end{remark}

\smallskip\noindent{\it Proof of Theorem \ref{thm3.16}}.
There exists a continuous  $\cal H_2$-potential $0<q_0\le q$, hence
$V(q_0)$ is $\cal H_1$-superharmonic (because $q$ is the supremum of an increasing sequence
of continuous potentials $q_n$, $n\in \NN$). By replacing $q$ by $q_0$  we may assume that the
$\cal H_2$-potential $q$ is continuous.
For any $U\in \cal U_c$
we have $V(q)=V(1_Uq)+V(1_{\Omega\smallsetminus U}.q)$. It follows from the
hypotheses of the theorem that $V(1_U.q)$ is continuous on $\Omega$ and
$V(1_{\Omega\smallsetminus U}q)$ is $\cal H_1$-harmonic in $U$  (by Remark \ref{remark4.6}),
hence $V(q)$ is continuous
on $U$, and consequently $V(q)$ is continuous on $\Omega$. For any $U\in \cal U_c$, the potential (in $U$)
$$p_U=V(q)-\widehat R_{V(q)}^{1,\Omega\smallsetminus U}=V(q)-H^1_U(V(q))$$
is finite and continuous on $U$. The family $(p_U)_{U\in {\cal U}_c}$ is a section of
continuous and finite $\cal H_1$-potentials on $\Omega$. Let $(\Omega,\cal H)$ be the biharmonic
space associated with the harmonic spaces $(\Omega,\cal H_1)$ and $(\Omega, \cal H_2)$ and
the section of continuous potentials $(p_U)_{U\in \cal U_c}$ by \cite{Bouk}.

It remains to show that the coupling kernel of $(\Omega,\cal H)$ is equal to $V$. We may assume
that the constant 1 is  $\cal H_2$-harmonic on $\Omega$ and that $V1$ is a
$\cal H_1$-surharmonic, hence a $\cal H_1$-potential on $\Omega$. Let $V'$ be the coupling kernel of biharmonic
space $(\Omega,\cal H)$. Since $V1$ is a $\cal H_1$-potential and
the pair $(V1,1)$ is  $\cal H$-harmonic in $\Omega$, then, according
to the Corollary \ref{cor3.10}, $V1$ is the pure  hyperharmonic
function of order 2 associated to the constant function $u\equiv 1$, hence $V'1=V1$. According to
the conditions on $V$ and properties of the kernel $V'$ we have $V=V'$
following \cite[Exercice 8.2.3., p. 198]{CC}.
\hfill$\square$

Let $(\Omega,\cal H)$ be a strong biharmonic
space with associated harmonic spaces
$(\Omega,\cal H_1)$ and $(\Omega,\cal H_2)$, and suppose that the harmonic space $(\Omega,\cal H_2)$ is a Green space,
and let
$y\mapsto q_y$ be a continuous map from $\Omega$ into $\cal S_2(\Omega)$ endowed with the
topology of R.-M. Herv\'e, which associates to every $y\in \Omega$ an
$\cal H_2$-potential $q_y$ with harmonic support $\{y\}$, and denote by $G_2$ the function defined
on $\Omega^2=\Omega\times \Omega$ by $G_2(x,y)=q_y(x)$.
If for any $y\in \Omega$ the function $p'_y$ (the pure hyperharmonic function associated with $q_y$)
is $\cal H_1$-superharmonic, then the function
$H$ defined on $\Omega^2$ by $H(x,y)=p'_y(x)$ is called a (or the) \textit{biharmonic Green kernel}
or \textit{second Green kernel}
of the biharmonic space $(\Omega,\cal H)$.\\

\section{Examples}\label{section7}

In this section we  give some examples showing  that the basic
results of  \cite{S3,S4} are incomplete. These results were used
in \cite{S5} to define   the adjoint biharmonic
space of a Green biharmonic space. Although the results of \cite{S5} are  true,
they  are built on  results that seem to be incomplete as mentioned in the introduction and
the previous sections. We shall clarify these aspects below.

\begin{example}\label{example4.1} Let  $N$ be an integer $\ge 1$, and $(\RR^N, \cal H)$
the biharmonic space of the example \ref{example2.1} corresponding to
$L_1=L_2=\Delta$.  The space $(\Omega,\cal H)$ is strong if and only
if $N\ge 5$ (see for example \cite[p. 588]{EK1}). Suppose that $N\ge
5$ and let $(p,q)$ be a  continuous $\cal H$-potential on $\RR^N$
such that
$p>0$ and $q>0$. The coupling kernel of $(\RR^N,\cal H)$ is  given
by
$$Vf(x)=c_N\int_{\RR^n}\frac{f(y)}{||x-y||^{N-2}}d\lambda(y)$$ for any function
 $f\in \cal B_+(\RR^N)$ and any $x\in \RR^N$, where $\lambda$ denotes the Lebesgue measure on
  $\RR^N$ and $c_N$ is a normalizing  constant satisfying
$\Delta \frac{c_N}{||\centerdot-y||^{N-2}}=-\epsilon_y$ in the distribution
sense for all $y\in \RR^N$, and where $\epsilon_y$ is the Dirac
measure at the point  $y$ ($G(x,y)=\frac{c_N}{||x-y||^{N-2}}$ is
the normalized Green kernel of $\RR^N$). Note first that there is no  biharmonic
pair $(h,k)$ with $h,k>0$ on $\RR^N$. Indeed, if a such  pair exists,
the function $k$ would be harmonic and  $>0$ in $\RR^N$, then constant,
say $k=c>0$. We would then have $h(x)\ge
Vk(x)=c.c_N\int_{\RR^N}\frac{1}{||x-y||^{N-2}}d\lambda(y)=+\infty$
for any $x\in \RR^n$,
which is a contradiction.

Consider now the sheaf ${\cal H'}$ of vector spaces of
pairs of real continuous functions defined for any open subset $\omega$ of
$\RR^N$ by

$${\cal H'}(\omega)=\{(h,k)\in {\cal C}(\omega)\times {\cal C}(\omega): (ph,qk) \ {\rm is}
\ \cal H{\rm -biharmonic \ on} \ \omega\},$$
so that we necessarily have

$${\cal H'}^*(\omega)=\{(u,v)\in {\cal C}_l(\omega)\times {\cal C}_l(\omega): (pu,qv) \ {\rm is}
\ \cal H-{\rm hyperharmonic \ on} \ \omega\},$$
where $\cal C_l(\omega)$ denotes the convex cone of nonnegative l.s.c. functions on $\omega$.

It is clear that $(1,1)\in {\cal H'}^*(\RR^N)$. But there is no
biharmonic pair $(u,v)$ on $\RR^N$ such that $u>0$ and
$v>0$. 
This example shows
that the claims a) and b), in  Remark 3 at page
323 of \cite{S4} are not true in general.
\end{example}

\begin{example}\label{example4.2}  This example was given by the second
author in \cite{EK0} and \cite{EK1}. Let $\Omega=[0,1[$ endowed by the topology induced by the usual one of $\RR$.
For any open subset $\omega$ of $\Omega$, we set

$${\cal H}_1(\omega)=\{u\in {\cal C}^2(\omega): (xu)"=0\}.$$
and
\begin{eqnarray*}
{\cal H}_2(\omega) = \{u\in {\cal C}(\omega): u(x)
=\frac{a}{x^2}+b \ {\rm on \ each \ connected \ component \ of} \
\omega, \ a,b\in \RR\},
\end{eqnarray*}
where ${\cal C}(\omega)$, resp. ${\cal C}^2(\omega)$, denotes the space
of continuous real functions, resp. real functions of class ${\cal C}^2$,  on
$\omega$ (if $0 \in \omega$ then each $u\in {\cal H}_2(\omega)$ is constant in the
component of 0).

Then we define a biharmonic sheaf on $\Omega$ by putting
$${\cal H}(\omega)=\{(u,v)\in \cal C^2(\omega): (xu)"=-v, v\in {\cal H}_2(\omega)\}$$
for any open subset $\omega$ of $\Omega$. It is easy to check
that $(\Omega,{\cal H})$ is a Brelot biharmonic space whose associated
harmonic spaces are $(\Omega,\cal H_1)$ and $(\Omega,\cal H_2)$.

Let $p$ and $q$ be the functions defined on $\Omega$ by
$p(x)=\frac{1}{2}-\frac{x}{2}$ and $q(x)=\min(1,v_0(x))$ where
$$v_0(x)=
\left\{\begin{array}{cc}
\frac{1}{x^2}-1 & {\rm if} \ x>0\\
+\infty & {\rm if} \ x=0.
\end{array}
\right.$$

The function $v_0$ is an $\cal H_2$-potential. Indeed, $v_0$ is an
$\cal H_2$-harmonic function on $]0,1[$ and, for any $\alpha\in
]0,1[$, if $k$ is a continuous function on $[0,\alpha]$,  $\cal
H_2$-harmonic on $\omega=[0,\alpha[$ and such that $k(\alpha)\le
v_0(\alpha)$, then   $k$ is a  constant function on $[0,\alpha]$ and
then we have  $k\le v_0$ on $\omega$. Hence $v_0$ is $\cal
H_2$-superharmonic $>0$ on $\Omega$. Let $h$ be a  $\cal H_2$-harmonic
function on $[0,1[$ such that $h\le v_0$, then  $h$ is necessarily a
constant function $\le 0$, thus $v_0$ is an $\cal H_2$-potential.

One can also easily check that the pair $(p,q)$ is a ${\cal H}$-potential.
Thus the biharmonic space $(\Omega, {\cal H})$ is strong. The
harmonic spaces  $(\Omega, {\cal H}_1)$ and $(\Omega, {\cal H}_2)$
are Green spaces,
however, $\{0\}$ is not the biharmonic support of any (extremal) pure
$\cal H$-potential. Indeed, $v_0$ is, up to the multiplication by some constant $>0$, the unique ${\cal H}_2$-potential of
harmonic support $\{0\}$. If the pure hyperharmonic function $u$ of order 2 associated
with  $v_0$ is ${\cal H}_1$-superharmonic, then the
pair $(u,v_0)$ would be ${\cal H}$-harmonic on $]0,1[$ according to
the Proposition \ref{prop3.8}, and then we would have $(xu)"(x)=-\frac{1}{x^2}+1$,
hence necessarily
$$u(x)=\frac{\ln(x)}{x}+\frac{x}{2}+a+\frac{b}{x},$$
where $a$ and $b$ are a real constants. But this is impossible since
 $u(x)<0$ for $x$ close to  $0$.

Let $(P,Q)$ be a continuous  ${\cal H}$-potential  on $\Omega$,
with $Q>0$ (and then $P>0$) in $\Omega$. Consider the biharmonic
sheaf ${\cal H}'$ on $\Omega$ defined by
$${\cal H}'(U)=\{(h,k)\in {\cal C}(U)\times {\cal C}(U):
(Ph,Qk)\in {\cal H}(U)\}$$
for any open subset $U$ of $\Omega$. Then $(1,1)\in {\cal H'}^*(\Omega)$, but
there is no pure  ${\cal H}'$-potential
of biharmonic support $\{0\}$. This
example shows that Theorem 2 of \cite[p. 321]{S4} is not true in general.
\end{example}

Let $(\Omega, \cal H)$ be a strong biharmonic space. We have seen that,
unlike the harmonic spaces admitting a potential, the set of all
points of $\Omega$ which are not the biharmonic support of any  extremal pure
potential is not always empty. In the case where $(\Omega,
{\cal H}_2)$ is a Green harmonic space, it was shown in \cite{S3} that this set is
nowhere dense (that is of empty interior)
in $\Omega$ (cf. \cite[Proposition 3.4]{S3}). In \cite{EK1}, the second author
has shown the  much more precise and general following result:

\begin{theorem}[{\cite[Theorem 3.1]{EK1}}]\label{thm4.3} Let $(\Omega,{\cal H})$
be a strong biharmonic space. Then the set of
points of $\Omega$ which are not the biharmonic support of extremal pure
potentials is ${\cal H}_2$-polar.
\end{theorem}

The proof of this result is based on the integral representation in
the cone of nonnegative ${\cal H}$-superharmonic pairs (see
\cite{EK1}).

In \cite{S4} it is also studied the existence and the regularity of
the biharmonic Green kernel (\cite[Proposition 6]{S4}), however it seems also
that his result is not true in general.

A Green harmonic space $(\Omega,\cal K)$ is said to be \textit{symmetric} if there is a
continuous function $y\mapsto p_y$ from $\Omega$ onto the cone of $\cal
K$-superharmonic functions $\ge 0$ endowed with the   topology of R.-M.
Herv\'e such that for every $y\in \Omega$, the function $p_y$ is a
$\cal K$-potential of harmonic support $\{y\}$ and $p_y(x)=p_x(y)$ for any
pair $(x,y)\in \Omega^2$ (then we say that the kernel-function
$G(x,y)=p_y(x)$, the Green kernel of $(\Omega,\cal K)$, is
symmetric).

\begin{theorem}[{\cite[Proposition 3.3]{EK1}}]\label{thm3.10}
Let $(\Omega,\cal H)$ be a strong Green biharmonic space
whose associated harmonic spaces  $(\Omega, \cal H_1)$ and $(\Omega,\cal
H_2)$ are symmetric and equal and, for any $y\in \Omega$, let $q_y$
be a $\cal H_2$-potential with harmonic support $\{y\}$. Then for any $y\in
\Omega$, the pure hyperharmonic function  $p'_y$  of order 2
associated with $q_y$ is superharmonic.  Moreover, if the function
$G$ defined by $G(x,y)=q_y(x)$ is symmetric, then the function
$H(x,y)=p'_y(x)$ is symmetric.
\end{theorem}

If the harmonic spaces  $(\Omega,\cal
H_1)$ and $(\Omega,\cal H_2)$ are symmetric but not equal, the
function $p'_y$  is not always a  $\cal H_1$-superharmonic
function for any $y\in \Omega$, contrary to the claim
in \cite{S4}, which is not true
in general as it is shown by the following example:

\begin{example} Let $(\Omega,\cal H)$ be the biharmonic space of the example
\ref{example4.2}. For any $y\in \Omega$ consider the functions
$$p_y(x)=
\left\{\begin{array}{c}
\frac{1}{y}-1 \ {\rm if} \ 0\le x\le y\\
\frac{1}{x}-1 \ {\rm if} \ y\le x<1,
\end{array}
\right.
$$
and
$$q_y(x)=
\left\{\begin{array}{c}
\frac{1}{y^2}-1 \ {\rm if} \ 0\le x\le y\\
\frac{1}{x^2}-1 \ {\rm if} \ y\le x<1,
\end{array}
\right.
$$
with the convention $\frac{1}{0}=+\infty$. Then, for each $y\in \Omega$, every  $\cal
H_1$-potential, resp. $\cal H_2$-potential, of harmonic support $\{y\}$ is
proportional to $p_y$, resp. $q_y$. Put $G_1(x,y)=p_y(x)$ and
$G_2(x,y)=q_y(x)$ for any pair $(x,y)\in \Omega^2$. The functions
$G_1$ and $G_2$ are continuous on $\Omega\times \Omega$ and, for any
 $y\in \Omega$ and any $i=1,2$, the function $G_i(.,y)$ is an $\cal
H_i$-potential $\cal H_i$-harmonic on $\Omega\smallsetminus\{y\}$. Thus
$G_1$ and $G_2$ are  (the) Green kernels of the harmonic
spaces $(\Omega,\cal H_1)$ and $(\Omega,\cal H_2)$ respectively.  It is clear that
$G_1$ and $G_2$ are symmetric. However the pure hyperharmonic
function of order 2 associated with $q_0$ is identically equal to $+\infty$. This example shows also that if the
spaces $(\Omega,\cal H_1)$ and $(\Omega,\cal H_2)$ are symmetric but
not necessarily equal, the biharmonic Green kernel is not necessarily
defined.
\end{example}

\section{The adjoint biharmonic space}\label{section8}

The  notion of adjoint harmonic functions   in a Green harmonic space
was introduced and
studied by R.-M. Herv\'e in \cite[Chap. VI]{He} to which we refer for the definition
and the properties of these functions.

An attempt
to consider and study the adjoint theory for
biharmonic functions was made in \cite{S5}.
However, the results of  \cite{S5} are based on
incomplete results as
we have  pointed out through the examples \ref{example4.1} and
\ref{example4.2}.  In this section we shall show that we can define correctly and
precisely a theory of  adjoint biharmonic space of a Green biharmonic space
$(\Omega,\cal H)$ in the case  where
this space has a "regular" biharmonic Green kernel (condtion ({\bf 6.3}) below). We
shall also show that this last condition is necessary for the existence of the biharmonic
adjoint of $(\Omega,\cal H)$.

The incomplete construction of the adjoint biharmonic space  in
\cite{S5} is based on the Theorem 2 of \cite{S4} which is not true
in general according to the example \ref{example4.2} of Section \ref{section7}.
To define correctly the adjoint biharmonic space, we shall proceed
in a different way.

In this section $(\Omega,\cal H)$  is a Green biharmonic space and
the c\^one $\cal S_j^+(\Omega)$, $j=1,2$, is equipped with the topology of R.-M. Herv\'e
(see \cite[Chap. IV]{He}).
For any $y\in \Omega$, we consider an $\cal H_1$-potential $p_y$ and
an $\cal H_2$-potential $q_y$ of harmonic support $\{y\}$ such that the
functions $y\mapsto p_y$ and $y\mapsto q_y$ from $\Omega$ into $\cal
S_1^+(\Omega)$ and $\cal S_2^+(\Omega)$ respectively are continuous
(cf. \cite[Theorem 18.1 and Proposition 18.1, p.p. 479-480]{He}),
and we denote by $p'_y$ the pure $\cal H_1$-hyperharmonic function
associated with $q_y$.

\begin{prop}\label{prop6.1}
The following assertions are equivalent:

{\rm 1.} The map $y\mapsto p'_y$ is continuous on $\Omega$.

{\rm 2.} For any $x\in \Omega$, the function $y\mapsto p'_y(x)$ is continuous
on $\Omega\smallsetminus \{x\}$.
\end{prop}

\begin{proof}

1.$\Rightarrow$2.
Suppose that the map $y\mapsto p'_y$ is continuous on $\Omega$
and let $x\in \Omega$ and $(y_n)$ a sequence of points
$\Omega\smallsetminus\{x\}$ converging to $y\in \Omega\smallsetminus \{x\}$.
Then we have $\lim\widehat{\inf} p'_{y_n}=p'_y$
and $\lim\widehat{\inf} q_{y_n}=q_y$, hence there exists $x_0\in \Omega\smallsetminus\{y\}$
and a subsequence $(y_{n_k})$ of $(y_n)$ such that $\lim p'_{y_{n_k}}(x_0)=p'_y(x_0)$
and $\lim q_{y_{n_k}}(x_0)=q_y(x_0)$. Then it follows from
\cite[Proposition 1.8, p. 301]{S6} and Ascoli-Arzela's Theorem that there is a subsequence
$(p_j,q_j)$ of $((p_{y_,'n_k},q_{n_k}))$ which converges locally uniformly
on $\Omega\smallsetminus \{y\}$ to an $\cal H$-superharmonic pair $(p,q_y)$, where $p$ is
a continuous function on $\Omega\smallsetminus \{y\}$.
Hence $p=\lim\widehat{\inf}p_j$ on $\Omega\smallsetminus\{y\}$ and therefore on $\Omega$.
Since the pair $(\lim\widehat{\inf}p_j,q_y)$ is $\cal H$-superharmonic on $\Omega$,
we have $p\ge p'_y$ on $\Omega\smallsetminus\{y\}$. It follows that
$p=p'_y$ since $p-p'_y$ is a nonnegative $\cal H_1$-superharmonic
on the domain $\Omega\smallsetminus \{y\}$ and vanishes  at $x_0$. In particular
we have $p(x)=p'_y(x)$, so that $\lim p_j(x)= p'_y(x)$. It follows that
for any $x\in \Omega\smallsetminus \{y\}$,
the sequence $(p'_{y_n}(x))$ converges to $p'_y(x)$. This proves the assertion 2.

2.$\Rightarrow$1. Let $x\in \Omega$ and suppose that the function $y\mapsto p'_y(x)$ is
continuous on $\Omega\smallsetminus\{x\}$. Let
$(y_n)$ be a sequence of points of $\Omega$ converging to $y$
and
$(y_{n_k})$ a subsequence of $(y_n)$ such that the sequence
$(p'_{y_{n_k}})$ converges in $\cal S_1^+(\Omega)$, we have $\lim\inf p'_{y_{n_k}}(x)=p_y(x)$ .
On the other hand the pair $(\lim\widehat{\inf}p'_{y_{n_k}}, q_y)$ is $\cal H$-hyperharmonic
on $\Omega$, hence $\lim\widehat{\inf}p'_{y_{n_k}}\ge p'_y$. It follows then that
$\lim\widehat{\inf}p'_{y_{n_k}}=p'_y$ on $\Omega\smallsetminus\{x\}$, hence on all of $\Omega$.
We conclude that the map $\mapsto p'_y$ is continuous on $\Omega$.
\end{proof}

The next proposition and its corollary give some concrete conditions under
which the map $y\mapsto p'_y$ is continuous.

\begin{definition}
A (function) kernel  on $\Omega$ is a l.s.c.
function $G: \Omega\times \Omega \rightarrow [0,+\infty]$.
The kernel $G$ on $\Omega$ is said to be regular
if, for any measure $\mu$ on $\Omega$ with compact support $K$, the function $G\mu$
is continuous on $\Omega$ if its restriction to $K$ is continuous.
\end{definition}

In a Green harmonic space $(X,\cal K)$, it follows by \cite[p. 521 and Th\'eor\`eme 25.1, p. 522]{He},
that the Green kernel of $X$ is regular if and only if $(X,\cal K)$ satisfies the axiom D
(axiom of domination, see \cite[Chapitre V]{He}).

In the following we denote by $G$ the kernel defined on $\Omega\times \Omega$ by $G(x,y)=p_y(x)$.

\begin{prop}\label{prop6.2}
Suppose that for any $y\in \Omega$, the  function $p'_y$
associated with $q_y$ is superharmonic, and that the kernel $G$ is regular.
Then the map $y\mapsto p'_y$ is continuous from $\Omega$ into $\cal S_1^+(\Omega)$.
\end{prop}

\begin{proof}
We may suppose that $(1,1)$ is a pure  $\cal H$-potential.
By Proposition \ref{prop3.4} the coupling kernel $V$ is given by $Vf(x)=\int G(x,y)f(y)d\mu(y)$
for any $f\in \cal B^+(\Omega$ and $x\in \Omega$, where $\mu$ is the (unique)
Radon measure $\ge 0$ on $\Omega$ such that $1=\int G(\centerdot,y)d\mu(y)$. Let $(y_n)$ be a sequence of points of $\Omega$ converging to
$y\in \Omega$. For any  $\varphi\in \cal C_c^+(\Omega)$, the function
$V(\varphi)$ is continuous, hence the sequence of measures
$q_{y_n}.\mu$ on $\Omega$ converges weakly to $q_y.\mu$.
Now we have $p'_{y_n}=G\mu_n$ for every integer $n$, and hence
$\lim p'_{y_n}=\lim{\widehat{\inf}} G\mu_n=G\mu=p'_y$ by
\cite[Theorem 3', p. 40]{Br}.
\end{proof}

\begin{cor}\label{cor6.3}
If the harmonic space $(\Omega,\cal H_2)$ satisfies the axiom of domination (or Axiom D,
see \cite[Chapitre V]{He}), then

{\rm 1.} The map $y\mapsto p'_y$ from $\Omega$ into $\cal S_2^+(\Omega)$ is continuous.

{\rm 2.} For any $x\in \Omega$, the function $y\mapsto p'_y(x)$ is continuous on $\Omega\smallsetminus \{x\}$.
\end{cor}

\begin{proof}
If the harmonic space $(\Omega,\cal H_2)$ satisfies the axiom D, then kernel
$G$ is regular as mentioned above, hence the assertion 1. and 2. hold
according to Proposition \ref{prop6.1} and Proposition \ref{prop6.2}.
\end{proof}


To consider  an adjoint  space, we shall proceed
differently to \cite{S5}. However, to do this, we need supplementary assumptions
on  the given biharmonic space $(\Omega,\cal H)$:


{\bf 6.1.} There is a base $\cal V$ of open subsets of $\Omega$
that are $\cal H_i$-completely determining for all $i=1,2$ (see Definition
of completely determining open subset of a Brelot harmonic space
in \cite[no 11, p. 449]{He}).

{\bf 6.2.} We suppose that the pair $(1,1)$ is a finite and
continuous $\cal H$-potential (which is not really a restriction
since we can always reduce the situation to this case by considering the
$(f,g)$-harmonic pairs, where $(f,g)$ is a continuous $\cal H$-potential
pair $>0$ on $\Omega$).\\

Let  $V$ be the coupling kernel of  $(\Omega,\cal
H)$. We denote by $\mu$ the measure
representing the pure hyperharmonic function  $V1$ of order 2
associated with the  $\cal H_2$-superharmonic constant function which equals 1, that is:
$$V1=\int p_yd\mu(y).$$
We recall that, according to the Proposition \ref{prop3.9}, the kernel
$V$ is given by
$$Vf(x)=\int p_y(x)f(y)d\mu(y), \forall x\in \Omega \ {\rm and} \ \forall f\in \cal B_+(\Omega).$$

We shall couple the biharmonic spaces $(\Omega,^*\cal H_2)$ and $(\Omega,^*\cal H_1)$ so that
in the obtained biharmonic space, the function $p'_\centerdot(y)$ is the pure hyperharmonic function
of order 2 associated with the $^*\cal H_1$-potential $p_y^*$
for any $y\in \Omega$. To do this we shall use the coupling defined in the
proof of Theorem \ref{thm3.16} relatively to the kernel $V^*$ defined by
$$V^*f(x)=\int q_x(y)f(y)d\mu(y), \ \forall f\in \cal B_+(\Omega),$$
so that we have, in particular,

\begin{equation}\label{eq7.1}
V^*(p^*_y)(x)=V^*(p_\centerdot(y))(x)=\int p_x(z)q_z(y)d\mu(z)=p'_x(y), \ \forall x,y\in \Omega.
\end{equation}
The kernel $V^*$ will be called the ``adjoint" kernel associated with $V$.


The conditions {\bf 6.1}, {\bf 6.2}
mean that the kernel $V^*$ satisfies some of  the hypotheses of
Theorem \ref{thm3.16}. In order $V^*$
is the coupling kernel, it needs to satisfy the following condition:

{\bf 6.3.} For any function $\varphi\in \cal C_c^+(\Omega)$, the
function $V^*\varphi$ is finite and continuous on $\Omega$.

We do not know if the condition {\bf 6.3.} is always satisfied
(under the hypotheses of the present section).
However we have the following


\begin{prop}\label{prop8.1}
If for any $y\in \Omega$, the function $p'_y$ is superharmonic and
if the function $y\mapsto p'_y$ from $\Omega$ into $\cal
S_2^+(\Omega)$ is continuous, then the  condition {\bf 6.3} is
satisfied.
\end{prop}

\begin{proof}
Let $\varphi\in \cal C_c^+(\Omega)$, $\varphi\not\equiv 0$, and
$y\notin S(\varphi)$, the support of $\varphi$. Since the function $p_y^*$ is continuous
and $>0$ on
$S(\varphi)$, we can find $\alpha\ge 0$ such that $\varphi\le \alpha
p_y^*$, and therefore $\alpha p'_\centerdot(y)=V^*(\alpha p_y^*-\varphi)+V^*(\varphi)$.
Since the function $ p'_\centerdot (y)$ is finite and
continuous on a neighborhood $U$ of $S(\varphi)$ not containing $y$
(following Proposition \ref{prop6.1}), and that
$V^*(\alpha p_y^*-\varphi)$ and $V^*(\varphi)$ are l.s.c., then
$V^*(\varphi)$ is finite and
continuous on $U$.
On the other hand, we have $V^*\varphi(x)=\int
q_y^*(x)\varphi(y)d\mu(y)$, and therefore $V^*\varphi$ is $^*\cal
H_2$-harmonic on the complement of the support $\varphi.$
Hence $V^*\varphi$ is finite and continuous on $\Omega$.
\end{proof}

\begin{cor}
If the harmonic space $(\Omega,\cal H_2)$ satisfies the axiom of domination
then the  condition {\bf 6.3} is
satisfied.
\end{cor}

\begin{proof}
The corollary results easily from the above proposition and Corollary \ref{cor6.3}.
\end{proof}

One may ask if the condition {\bf 6.3} is satisfied by $V^*$ in the case
where there are points of $\Omega$ which are not the biharmonic support
of a pure  potential. The example \ref{example4.2} shows that the
answer is not true in general. To see this (in the example \ref{example4.2}) note first that
$(\frac{1-x}{2},1)$ is  a finite  continuous $\cal
H$-superharmonic pair on $[0,1[$. For any $y\in [0,1[$  we put
$$p_y(x)=
\left\{\begin{array}{c}
\frac{1}{y}-1 \ {\rm if} \ 0\le x\le y\\
\frac{1}{x}-1 \ {\rm if} \ y\le x<1,
\end{array}
\right.
$$
and
$$q_y(x)= \left\{\begin{array}{c}
\frac{1}{y^2}-1 \ {\rm if} \ 0\le x\le y\\
\frac{1}{x^2}-1 \ {\rm if} \ y\le x<1,
\end{array}
\right.
$$
with the  convention $\frac{1}{0}=+\infty$. It is clear
that the functions $y\mapsto p_y$ and $y\mapsto q_y$ from $\Omega$
into $\cal S_1^+(\Omega)$ and from $\Omega$ into $\cal S_2^+(\Omega)$
respectively, endowed with the topology of
R.-M. Herv\'e, are continuous. Let $\mu$ be the measure on $\Omega$ defined
by $\mu(dy)=y(1-y)\lambda(dy)$, where $\lambda$ denotes the Lebesgue measure
on $[0,1[$. A simple calculation which we do not detail here  gives
$\frac{1}{2}(1-x)=\int p_y(x)d\mu(y)$. The coupling kernel of
the harmonic spaces $(\Omega,\cal H_1)$ and $(\Omega,\cal H_2)$ in this
case is given by
$$Vf(x)=\int p_y(x)f(y)d\mu(y), $$
and thus
$$V^*f(x)= \int q_x(y)f(y)d\mu(y)$$
for any function  $f\in \cal B^+(\Omega)$ and any $x\in \Omega$.
Let $\varphi$ be a nonnegative continuous on
$[0,1[$ of compact support and such that $\varphi(0)>1$. Then we can
find a real $\alpha\in ]0,1[$ such that $\varphi \ge 1$ on
$[0,\alpha]$ so we have
$$
V^*(\varphi)(0) \ge \int_0^\alpha\frac{(1-y^2)(1-y)}{y}dy=+\infty,
$$ which proves that the condition {\bf 6.3} is not satisfied. We shall prove later
(Proposition \ref{prop7.3}
) the converse of
Proposition \ref{prop8.1}, that is, if the condition
{\bf 6.3} (of course in addition of {\bf
6.1}, {\bf 6.2})  is satisfied, then the
function $p'_y$ is $\cal H_1$-superharmonic for any $y\in \Omega$
and that the map $y\mapsto p'_y$ is continuous.\\

Now we are going to couple the harmonic spaces $(\Omega,{^*\cal H_2})$ and
$(\Omega, {^*\cal H_1})$ by mean of the kernel $V^*$, so that in the  obtained biharmonic space,
the function $p'_\centerdot(y)$ is the pure hyperharmonic function of order 2
associated with the  $^*\cal H_1$-superharmonic function $p^*_y=p_\centerdot(y)$ for
any  $y\in \Omega$ as we have seen above.

\begin{theorem}\label{thm7.1}
Let $(\Omega,\cal H)$ be a strong Green biharmonic space whose  associated
harmonic spaces   $(\Omega,\cal H_1)$ and $(\Omega,\cal H_2)$
satisfy  the conditions {\bf 6.1}, {\bf 6.2} and {\bf 6.3}. Then there is a unique
(Green) biharmonic space $(\Omega,{^*\cal H})$ whose associated harmonic spaces
are respectively $(\Omega,{^*\cal H}_2)$ and $(\Omega,{^*\cal H}_1)$
and such that for any $y\in \Omega$, the pure hyperharmonic of order
2 associated with $p_y^*$ is equal to $p'_\centerdot(y)$.
\end{theorem}

\begin{proof} The kernel $V^*$ satisfies the hypotheses of Theorem
\ref{thm3.16}, hence the existence (and uniqueness) of a biharmonic space $(\Omega, {^*\cal H})$ defined by
the coupling of the harmonic spaces $(\Omega,^*\cal H_2)$ and
$(\Omega,^*\cal H_1)$ by mean of this kernel. According to the
formula (6.1), for any $y\in \Omega$ the pure hyperharmonic function of order 2 associated
with the
$^*\cal H_1$-superharmonic function $p^*_y$ is equal to $p'_{\cdot}(y)$.
\end{proof}

The biharmonic space $(\Omega, {^*\cal H})$ in  Theorem
\ref{thm7.1} is called the \textit{adjoint biharmonic space}  of
$(\Omega,\cal H)$. For any open set $\omega\subset \Omega$, the pairs $(u,v)\in \cal H^*(\omega)$ are called
the \textit{adjoint biharmonic pairs} on $\omega$.

\begin{remark}\label{remark7.2}
The relatively compact completely determining open subsets  of the base $\cal V$ form a
base of open subsets which are both $^*\cal H_1$-regular and
$^*\cal H_2$-regular, hence $^*\cal H$-regular. Let $\omega$ be a
$^*\cal H$-regular (relatively compact) open subset
 and $\varphi$ a continuous function
$\ge 0$ on $\partial \omega$.
For any $x\in \omega$, the linear form
$\varphi\mapsto  V_\omega^*(^*H_\omega^1(\varphi))(x)$  on $\cal C(\partial\omega)$
defines a (Radon) measure $\tau_x^\omega$
on $\partial\omega$. Here $^*H_\omega^1(\varphi)$ is the solution of the Dirichlet problem
in $\omega$ in the harmonic
space $(\Omega,^*\cal H_1)$
for the data $\varphi$ on $\partial \omega$, and $V_{\omega}$ is
is the Borel kernel defined on $\omega$ by $V_{\omega}(f)=
V(\overline f)-\widehat R_{V(\overline f)}^{1,\Omega\smallsetminus \omega}$
for any $f\in \cal B^+(\omega)$, where $\overline f$ is the function on $\Omega$
which equals  $f$ on $\omega$ and $0$ on $\Omega\smallsetminus \omega$.
The triple of biharmonic measures of $\omega$ at the point $x$ is
$(\rho_x^\omega,\tau_x^\omega,\sigma_x^\omega)$, where
$\rho_x^\omega$ and $\sigma_x^\omega$ are respectively the harmonic
measures of $\omega$ at the point $x$ in the adjoint harmonic spaces
$(\Omega, {^*\cal H_2})$ and $(\Omega, {^*\cal H_1})$ (see
\cite[Chap. VI]{He}).
\end{remark}

Now we prove the converse of the Proposition \ref{prop8.1}:
\begin{prop}\label{prop7.3}
Suppose that there  exists a strong biharmonic space $(\Omega,\cal
G)$ whose associated harmonic spaces are $(\Omega,{^*\cal H_2})$ and
$(\Omega,{^*\cal H_1})$ and such that for any $y\in \Omega$, the
function $p'_\centerdot(y)$ is the pure hyperharmonic function of order 2
associated with $p^*_y=p_\centerdot(y)$. Then for any $y\in \Omega$, the
function $p'_y$ is $\cal H_1$-superharmonic
and the function $y\mapsto p'_y$ is continuous on $\Omega$.
\end{prop}

\begin{proof} Suppose that there exists a strong biharmonic space
$(\Omega,\cal G)$ whose associated  harmonic spaces  are
$(\Omega,{^*\cal H_2})$ and $(\Omega,{^*\cal H_1})$ and such that
for any  $y\in \Omega$, the function $p'_\centerdot(y)$ is the pure
hyperharmonic  function of order 2 associated with $p^*_y=p_\centerdot(y)$.
Assume that there exists $y\in \Omega$ such that $p'_y\equiv +\infty$ on
$\Omega$, then we have $p'^*_x(y)=p'_y(x)=+\infty$ for any $x\in
\Omega\smallsetminus (N\cup \{y\})$,
but this is absurd according to the
Theorem \ref{thm4.3}, since for  such $x$ we have
$p'_x(y)<+\infty$ (because the pair $(p'_\centerdot(x), p_x^*)$ is $\cal
G$-biharmonic on $\Omega\smallsetminus \{x\}$), where $N$ is the  $^*\cal
H_1$-polar subset of points $y$ of $\Omega$ such that the pure
hyperharmonic function of order 2 associated with $p^*_y$ is identically equal to $+\infty$.
Now let $x\in \Omega$, then the pair $(p'_\centerdot(x), q^*)$ is
$\cal H$-harmonic on $\Omega\smallsetminus \{x\}$ according to
Proposition \ref{prop3.8}, and hence the function $p'_\centerdot(x)$ is continuous
on $\Omega\smallsetminus \{x\}$. Let $z\in \Omega\smallsetminus\{x\}$ and
$\cal F$ an ultrafilter on $\Omega$ finer than the filter
of neighborhoods of $z$, then
we have $\lim\widehat{\inf}_{\cal F}p'_y(x)=p'_z(x)$,
and it follows that this equality holds also for $z=x$
(indeed, if $u$ and $v$ are $\cal K$-hyperharmonic on
$X$ in a harmonic space
$(X,\cal K)$ such that $u=v$ on $X\smallsetminus\{x\}$ for some
point $x\in X$, then $u=v$).
It follows then that the function $y\mapsto p'_y$ is continuous on $\Omega$.
The proposition is proved.
\end{proof}

Now we proceed to establish a converse of Theorem \ref{thm7.1}, that
is, the existence of an adjoint   biharmonic space $(\Omega,\cal H^*)$
associated with the  biharmonic space $(\Omega,\cal H)$ satisfying the
hypotheses of Theorem \ref{thm7.1} and the property that for any $y\in \Omega$, the pure
hyperharmonic function of order 2 associated with $p_y^*$ is
$p'_\centerdot(y)$, implies that the coupling kernel of $(\Omega,\cal G)$
is $V^*$, and hence $V^*$ necessarily satisfies the condition
{\bf 6.3}.

\begin{theorem}\label{thm7.4}
Suppose that
there exists a biharmonic space $(\Omega,\cal G)$ whose
associated harmonic spaces are respectively $(\Omega,{^*\cal H_2})$
and $(\Omega,{^*\cal H_1})$, such that, for any $y\in \Omega$, the pure
hyperharmonic function of order 2 associated with $p_y^*$ is equal to
$p'_\centerdot(y)$. Then the coupling kernel of $(\Omega,\cal G)$ is $V^*$.
\end{theorem}

\begin{proof} Let $W$ be the coupling kernel of $(\Omega,\cal G)$.
According to the hypotheses of the theorem,
 the biharmonic space
$(\Omega,\cal G)$ is strong. For any $y\in \Omega$, we have
$V^*(p_y^*)=W(p_y^*)=p'_\centerdot(y)$ by definition of $V^*$ and
$W$. Let $p$ be a $^*\cal H_1$-potential $>0$ on
$\Omega$ such that $Wp$ is finite and continuous on $\Omega$ (such a potential exists by
Corollary \ref{cor3.6}).
According to \cite[Theorem 18.2, p. 482]{He}, there
exists a measure $\mu$ on $\Omega$ such that $p=\int p_\centerdot(y)d\mu(y)$. Then
by applying twice the Fubini theorem,  we have for any $x\in
\Omega$,
\begin{eqnarray*}
Wp(x)&=&\int p(z)W(x,dz)
=\int\int p^*_y(z)d\mu(y)W(x,dz)\\
&=& \int\int p^*_y(z)W(x,dz)d\mu(y)
= \int W(p^*_y)(x)d\mu(y)\\
&=& \int V^*(p^*_y)(x)d\mu(y)
= V^*p(x),
\end{eqnarray*}
and hence $Wp=V^*p$. Since $W$ satisfies the condition 2 of Theorem \ref{thm3.7}, we
deduce that $V^*$ also satisfies this condition. Indeed, let
$\varphi\in \cal C_c^+(\Omega)$
$\alpha
>0$ such that $\varphi \le \alpha p$. Then $\alpha
V^*p=V^*(\varphi)+V^*(\alpha p-\varphi)$. The functions
$V^*(\varphi)$ and $V^*(\alpha p-\varphi)$ are both finite and
l.s.c. and $V^*p$ is finite and continuous on $\Omega$,
then $V^*(\varphi)$ is finite and continuous. Moreover,
by the definition of the kernel $V^*$, the function $V^*\varphi$ is
$^*\cal H_2$-harmonic in $\Omega\smallsetminus S(\varphi)$. Let $v\in
{^*\cal H_1^{*+}}(\Omega)$,
then according to the above results we have $Wv=\sup_nW(v\wedge
np)=\sup_nV^*(v\wedge np)=V^*v$.
It therefore follows from  Theorem \ref{thm3.7} that $W=V^*$.
\end{proof}

We have at the same time proved (under the assumptions of
Theorem \ref{thm7.4}) the following result:

\begin{cor}
Under the hypotheses of the previous theorem,
for any $\varphi\in \cal C_c^+(\Omega)$, the function
$V^*\varphi$ is a finite and continuous $^*\cal H_2$-potential on $\Omega$.
\end{cor}



\section{Application to  the study  of the regularity of the biharmonic Green kernel}\label{section9}

We keep the notations of the previous sections. Let
$(\Omega,\cal H)$ be a strong Green biharmonic space.
For any $y\in \Omega$, we consider an
$\cal H_2$-potential $q_y$ of harmonic support $\{y\}$ such that the
function $y\mapsto q_y$ from $\Omega$ into $\cal
S_1^+(\Omega)$ (endowed with the topology of R.-M. Herv\'e) is continuous
(cf. \cite[Theorem 18.1 and Proposition 18.1, p. 479-480]{He}).
Suppose that, for any $y\in \Omega$, the function $p'_y$ is
$\cal H_1$-superharmonic on $\Omega$,  and that the map
$y\mapsto p'_y$ is continuous from $\Omega$ into $\cal S_1^+(\Omega)$ (hence
for any $\varphi\in \cal C_c^+(\Omega)$, the function
$V^*\varphi$ is a finite and continuous $^*\cal H_2$-potential).
Put $H(x,y)=p'_y(x)$ for
any pair $(x,y)\in \Omega^2$. The kernel $H$ was called the biharmonic
Green kernel of the biharmonic space $(\Omega, \cal H)$ in the Section \ref{section2}.

\begin{theorem}\label{thm7.2} Suppose that the conditions ({\bf 6.1}) and ({\bf 6.2})
are satisfied and
that for any $y\in \Omega$, the function
$p'_y$ is $\cal H_1$-superharmonic on $\Omega$
and, for any $x\in \Omega$, the function $y\mapsto p'_y(x)$
is continuous on $\Omega\smallsetminus \{x\}$. Then the kernel $H$ has the following properties

{\rm 1.} $H$ is l.s.c. on $\Omega\times \Omega$.

{\rm 2.} $H$ is continuous on $\Omega\times \Omega\smallsetminus \Delta$, where
$\Delta=\{(x,x): x\in \Omega\}$ is the diagonal of $\Omega\times
\Omega$.
\end{theorem}

\begin{proof}  We  may suppose that the pair $(1,1)$ is $\cal
H$-superharmonic. The hypotheses of Theorem \ref{thm3.16} are satisfied,
and thus the adjoint biharmonic space $(\Omega,^*\cal H)$ is well defined.

1. We have $H(x,y)=\int p_z(x)q_y(z)d\mu(z)$ for any
pair $(x,y)\in \Omega^2$. On the other hand, for any $z\in \Omega$,
the functions $(x,y)\mapsto p_z(x)$ and $(x,y) \mapsto q_y(z)$ are
l.s.c. on $\Omega\times \Omega$, then it follows  from the   Fatou
lemma that the function $H$ is l.s.c. on $\Omega\times \Omega$.

2. Let  $((x_n,y_n))$ be a sequence of points of $\Omega\times
\Omega\smallsetminus \Delta$ which converges to $(x,y)\in \Omega\times
\Omega\smallsetminus \Delta$, and  $\omega$ an open neighborhood of
$x$ such that $y\notin{\overline \omega}$. We can obviously assume
that $y_n\notin {\overline \omega}$ for any integer $n$. The pairs $(H(x,\centerdot), q_\centerdot(x))$
are ${^*\cal H}$-biharmonic, hence continuous on $\Omega\smallsetminus\{x\}$.
Thus the sequences $(H(x,y_n))$ and $(q_{y_n}(x))$ are bounded. On the other hand, the pairs
$(H(\centerdot,y_n), q_{y_n})$, $n\in \NN$, are  biharmonic on $\omega$, then,
according to  \cite[Theorem 3.2, p. 586]{EK1}, the sequences of
functions $(H(\centerdot,y_n))_{n\in \NN}$, $(q_{y_n})_{n\in \NN}$, are
equicontinuous at  $x$. Let $\epsilon>0$, we can find a
neighborhood $W$ of $x$ such that $W\subset \omega$ and
$|H(\xi,y_n)-H(x,y_n)|<\epsilon$  for  any $\xi\in W$ and any $n\in
\NN$. By applying the same arguments in he biharmonic space
$(\Omega,^*\cal H)$ to the kernel $H^*$
(defined by $H^*(x,y)=H(y,x)$), there is
an open neighborhood $W'$ of $y$ such that $|H(x_n,
\xi)-H(x_n,y)|<\epsilon$ for any $\xi\in W'$ and any $n\in \NN$.
Furthermore, there exists an integer $n_0\in \NN$ such that
$x_n\in W$ and $y_n\in W'$ for any $n\ge n_0$. Thus
$$|H(x_n,y_n)-H(x,y)|\le
|H(x_n,y_n)-H(x_n,y)|+|H(x_n,y)-H(x,y)|<2\epsilon.$$
for every $n\ge n_0$. We then deduce that
the sequence $(H(x_n,y_n))$ converges to $H(x,y)$. Hence the
function $H$ is continuous on $\Omega\times \Omega\smallsetminus \Delta$.
\end{proof}

\end{document}